\documentclass[12pt,twoside,reqno]{amsart}
\usepackage{pifont}
\usepackage{amsmath}
\usepackage{amsthm}
\usepackage{bm}
\usepackage{amsfonts}
\usepackage{amssymb}
\usepackage{latexsym}
\usepackage{amstext}
\usepackage{array}
\usepackage{graphicx}
\usepackage{tikz}
\usepackage{mathrsfs}
\usepackage{diagbox}
\usepackage{threeparttable}
\usepackage{extarrows}

\date{}
\allowdisplaybreaks[4] \footskip=15pt
\renewcommand{\uppercasenonmath}[1]{}

\newtheorem{thm}[subsection]{Theorem\;}
\newtheorem{cor}[subsection]{Corollary\;}
\newtheorem{Def}[subsection]{Definition\;}
\newtheorem{lem}[subsection]{Lemma\;}
\newtheorem{remark}{Remark\;}

\newtheorem{prop}[subsection]{Proposition\;}
\newtheorem{exm}[subsection]{Example\;}
  
\newcommand{\bthm}{\begin{thm} }
	\newcommand{\ethm}{\end{thm} }
\newcommand{\bpro}{\begin{prop}}
	\newcommand{\epro}{\end{prop}}
\newcommand{\bdf}{\begin{Def}}
	\newcommand{\edf}{\end{Def}}
\newcommand{\bexm}{\begin{exm}}
	\newcommand{\eexm}{\end{exm}}

\newcommand{\blem}{\begin{lem}}
	\newcommand{\elem}{\end{lem}}
\newcommand{\bpf}{\begin{proof}}
	\newcommand{\epf}{\end{proof}}
\newcommand{\bcor}{\begin{cor}}
	\newcommand{\ecor}{\end{cor}}
\newcommand{\ba}{\begin{array}}
	\newcommand{\ea}{\end{array}}
\newcommand{\bea}{\begin{eqnarray}}
	\newcommand{\eea}{\end{eqnarray}}

\newcommand{\brem}{\begin{remark}}
	\newcommand{\erem}{\end{remark}}

\def\Z{\mathbb{ Z}}

\def\deg{\mathrm{deg}}

\graphicspath{dir-list}
\begin{document}
	\begin{center}
		{\large  \bf  
		Quantitative Estimates for the Size of the Zsigmondy Set in Arithmetic Dynamics
		}
		\footnote {Supported by NSFC (Nos. 12071209, 12231009).}\\
		\vskip 0.8cm
		{\small   Yang Gao and Qingzhong Ji \footnote{Corresponding author.\\ \indent E-mail addresses: 864157905@qq.com (Y. Gao),\; qingzhji@nju.edu.cn (Q. Ji)}}\\
		{\small School of Mathematics, Nanjing University, Nanjing
			210093, P.R.China}
	\end{center}

{\bf Abstract }
Let \( K \) be a number field. We provide quantitative estimates for the size of the Zsigmondy set of an integral ideal sequence generated by iterating a polynomial function \(\varphi(z) \in K[z]\) at a wandering point \(\alpha \in K.\)


\section{\bf Introduction and main results}\label{02}
Let $\mathcal{A}=\left\{A_n\right\}_{n \geqslant 1}$ be a sequence of integers. A prime $p$ is called a primitive divisor of $A_n$ if
            $
		p \mid A_n$  and $p \nmid A_i $  for all $1 \leqslant i<n.$
		The Zsigmondy set of $\mathcal{A}$ is the set
		$
		\mathcal{Z}(\mathcal{A})=\left\{n \geqslant 1\;|\;A_n \text { does not have a primitive divisor }\right\}.
		$
Studying the finiteness of Zsigmondy sets has a long history in number theory.	
		In 1892, Zsigmondy \cite{Zsi} proved that if $a$ and $b$ are relatively prime positive integers and $ab\neq 1,$ then the Zsigmondy set $\mathcal{Z}\left(\left\{a^n - b^n\right\}_{n \geqslant 1}\right)$ is finite. 
In 2006, Everest, McLaren, and Ward \cite{EGG} proved the finiteness of the Zsigmondy set for the elliptic divisibility sequence $\{d_n\}_{n \geqslant 1}.$ 
In 2007, Rice \cite{Rice} studied polynomial recursive sequences and proved that the Zsigmondy set is finite for the sequence $\{a_n\}_{n \geqslant 1}$ generated by $a_{n+1} = f(a_n),$ where $f(z) \in \mathbb{Z}[z].$
In 2009, Ingram and Silverman \cite{IP} extended Rice's results to rational functions over algebraic number fields. 
In 2011, Doerksen and Haensch \cite{DKH} examined the primitive divisors in the critical orbit of polynomials of the form $z^d + c \in \mathbb{Z}[z].$  For more information on this topic, please refer to  \cite{GNT}, \cite{Jones}, \cite{HOK}, \cite{LC2},  and \cite{RF}.

Let $K$ be a number field and let $\mathcal{A}=\left\{\mathfrak{A}_n\right\}_{n \geqslant 1}$ be a sequence of nonzero integral ideals of $K.$ A prime ideal $\mathfrak{p}$ is called a primitive divisor of $\mathfrak{A}_n$ if
		$
		\mathfrak{p} \mid \mathfrak{A}_n$ and $\mathfrak{p} \nmid \mathfrak{A}_i $ for all  $1 \leqslant i < n.$
  The Zsigmondy set of $\mathcal{A}$ is the set
 $$\mathcal{Z}(\mathcal{A}) =\left\{n \geqslant 1\;|\; \mathfrak{A}_n\text{ does not have a primitive divisor}\right\}.$$

In 2009, Ingram and Silverman \cite{IP} proposed a conjecture as follows.\\
\textbf{Conjecture.} Let $\mathcal{O}_K$ be the ring of algebraic integers of a number field $K.$ Let $\varphi(z) \in K(z)$ be a rational function of degree $d \geqslant 2,$ and let $\alpha \in K$ be a $\varphi$-wandering point. For each $n \geqslant 1,$ write the ideal
$$
\left(\varphi^n(\alpha)-\alpha\right) \mathcal{O}_K=\mathfrak{A}_n \mathfrak{B}_n^{-1}
$$
as a quotient of relatively prime integral ideals (If $\varphi^n(\alpha)=\infty,$ we set $\mathfrak{A}_n=(1),$  $\mathfrak{B}_n=(0)$). Then the dynamical Zsigmondy set $\mathcal{Z}\left(\left\{\mathfrak{A}_n\right\}_{n \geqslant 1}\right)$ is finite.

In \cite{Ji and Zhao}, the second author of this paper, joint with Z. Zhao, studied similar problems for a Drinfeld module.

In 2013, Holly Krieger \cite{HOK}, in her PhD thesis, introduced the notation of an \(S\)-rigid divisibility sequence.
\begin{Def}\label{520001} Let $K$ be a number field with \(\mathcal{O}_K\) the ring of algebraic integers of $K.$ Let \(S\) be a finite set of places, including all archimedean ones.   We say that a sequence  \(\left\{\mathfrak{a}_n\right\}_{n \geqslant 1}\)  integral ideals of \(\mathcal{O}_K\) is an \(S\)-rigid divisibility sequence if it satisfies the following conditions:
\begin{enumerate}
    \item For every \(\mathfrak{p} \notin S\) and all \(m, n \in\mathbb{N}^*,\) if\; \(\mathfrak{p} \mid \operatorname{gcd}\left(\mathfrak{a}_n, \mathfrak{a}_m\right),\) then \(\mathfrak{p} \mid \mathfrak{a}_{\operatorname{gcd}(m, n)}.\)
    \item For every \(\mathfrak{p} \notin S\) and  \(m \in \mathbb{N}^*\) with \(\operatorname{ord}_{\mathfrak{p}}\left(\mathfrak{a}_m\right)>0,\) we have \(\operatorname{ord}_{\mathfrak{p}}\left(\mathfrak{a}_{km}\right)=\operatorname{ord}_{\mathfrak{p}}\left(\mathfrak{a}_m\right)\) for all \(k \geqslant 1.\)
\end{enumerate}
\end{Def}

  Let  \( f \in K[z],\) $\operatorname{deg}f\geq 2$ and \( \alpha \in K \) with \( f'(\alpha) = 0 \) and an infinite forward orbit. Write
  $$
  \left(f^n(\alpha) - \alpha\right)\mathcal{O}_K = \mathfrak{a}_n \mathfrak{b}_n^{-1}
  $$
  with \( \mathfrak{a}_n \) and \( \mathfrak{b}_n \) coprime ideals. Holly Krieger {\rm(\cite{HOK})} proved that  the sequence \( \{ \mathfrak{a}_n \}_{n \geq 1} \) is an \( S \)-rigid divisibility sequence for some \( S \) and the Zsigmondy set $\mathcal{Z}\left(\left\{\mathfrak{a}_n\right\}_{n \geqslant 1}\right)$ is finite. Furthermore, if \( f(z) \in \mathcal{O}_K[z],\) one can choose \( S = M^{\infty}_K.\) 

In this paper, we prove the following results:
\begin{thm}\label{520002}
Let $K$ be a number field with \(\mathcal{O}_K\) the ring of algebraic integers of $K.$
  Let $\varphi(z)\in K[z]$ with $\operatorname{deg} \varphi \geqslant 3.$ Suppose $\alpha\in K$ is a wandering point of $\varphi(z).$ For each $n \geqslant 1,$ write the ideal $(\varphi^n(\alpha)-\alpha)\mathcal{O}_K = \mathfrak{A}_n \mathfrak{B}_n^{-1}$ as a quotient of relatively prime integral ideals. If the sequence of ideals $\{\mathfrak{A}_n\}_{n \geqslant 1}$ is an $S$-rigid divisibility sequence for some \(S,\) then the Zsigmondy set $\mathcal{Z}\left(\left\{\mathfrak{A}_n\right\}_{n \geqslant 1}\right)$ is finite and there is a constant $M>0$ depending only on $\operatorname{deg}\psi,$  $[K:\mathbb{Q}],$ $\#S,$  $h(\frac{1}{\psi(\frac{1}{z})}),$  $h(\psi),$ and $\hat{h}_{\psi}(0),$ such that 
$\#\mathcal{Z}(\{\mathfrak{A}_n\}_{n \geqslant 1})\leqslant M,$ where $\psi(z)=\varphi(z+\alpha)-\alpha.$
\end{thm}
\begin{remark}\label{3309}
Let \(d = \operatorname{deg} \psi.\) Specifically, we can take \(M\) as follows
$$ 1+\log_d^+ \left( \frac{8(c_3(d) + c_4(d) h(\psi))}{\widehat{h}_{\psi}(0)} \right)+\;\frac{8(c_3(d) + c_4(d)h(\psi))}{(3d-7) \hat{h}_{\psi}(0)}+\; 4^{\#S} \gamma + \log_d \left( \frac{h(\tilde{\psi})}{\hat{h}_{\psi}(0)} + 1 \right),$$
where $\log_d^{+} x = \log_d \max \{1, x\},$  \(\gamma\) depends only on \(d\) and \([K : \mathbb{Q}],\) and \(c_3(d)\) and \(c_4(d)\) are constants depending only on \(d.\) Moreover, explicit expressions for \(c_3(d)\) and \(c_4(d)\) in terms of \(d\) can be derived.

Here, \( h(\psi) \) is the height of the rational map \(\psi,\) \( h(\tilde{\psi}) \) is the height of \(\frac{1}{\psi\left(\frac{1}{z}\right)},\) and \(\hat{h}_{\psi}(0)\) is the canonical height of \(\psi\) at \(0.\) 
\end{remark}

\begin{exm}
   Let \( K \) be a number field with \(\mathcal{O}_K\) the ring of algebraic integers of $K.$
For the polynomial \(\varphi(z) = z^d + c\) with $d\geqslant 3$ and \(c \in \mathcal{O}_K,\) there exists a constant \(M > 0\) depending only on \(d\) and the degree \([K : \mathbb{Q}]\) such that 
   \(
   \#\mathcal{Z}\left(\left\{\varphi^n(0)\right\}_{n \geqslant 1}\right) \leqslant M.
   \)
    This follows from Remark {\rm\ref{3309}} and  Theorem 1 in \cite{PG1}.
\end{exm}

\begin{Def}
A polynomial $\varphi(z)\in K[z]$ is called a powerful polynomial over $K$ if $\operatorname{deg}(\varphi) \geqslant 2$ and \( p(z)^2 \) divides \( \varphi(z) \) for every irreducible factor \( p(z) \) of \( \varphi(z).\) 
\end{Def}
For example, $\varphi(z) = f_1(z)^{e_1} f_2(z)^{e_2} \cdots f_m(z)^{e_m}$ is powerful, 
where \( f_1(z),\dots, f_m(z) \in K[z] \)  (not necessarily pairwise distinct), \( e_i \geqslant 2 \) are integers, and \( \deg(f_i(z)) \geqslant 1 \) for each \( i = 1, \ldots, m.\)

\begin{thm}\label{520004} 
 Let \( K \) be a number field with \(\mathcal{O}_K\) the ring of algebraic integers in \( K.\)  
Suppose \( \varphi \in K[z] \) is a powerful polynomial with \(\deg \varphi \geq 3,\) and let $0$ be a wandering point of \( \varphi(z).\)
For each \(n \geqslant 1,\) write the ideal \((\varphi^n(0))\mathcal{O}_K = \mathfrak{A}_n \mathfrak{B}_n^{-1}\) as a quotient of relatively prime integral ideals. Then the sequence \(\{\mathfrak{A}_n\}_{n \geqslant 1}\) forms an \(S\)-rigid divisibility sequence for some \(S,\) and the Zsigmondy  set \(\mathcal{Z}(\{\mathfrak{A}_n\}_{n \geqslant 1})\) is finite. In particular, if \(\varphi(z) = f_1(z)^{e_1} f_2(z)^{e_2} \cdots f_m(z)^{e_m},\) where \(f_i(z) \in \mathcal{O}_K[z]\) for \(1 \leqslant i \leqslant m,\) then we can take \(S = M^{\infty}_K.\)
\end{thm}

\begin{thm}\label{520005}
Let $m \geqslant 2,$ and let $$\varphi(z) = (zf_1(z) + a_1)^{e_1} (zf_2(z) + a_2)^{e_2} \cdots (zf_m(z) + a_m)^{e_m}$$ where $a_1, a_2, \ldots, a_m$ are integers, with $|a_j| \geqslant 2$ for some $j.$ Assume that $e_i \geqslant 2$ and $f_i(z) \in \mathbb{Z}[z]$  has no integer roots for $i = 1, 2, \dots, m.$ \;Then
\begin{enumerate}
\item[(1)] $0$ is a preperiodic point of $\varphi(z)$ if and only if  $0$ is a fixed point of $\varphi(z).$
    \item[(2)] If $0$ is a wandering point of $\varphi(z),$ then the Zsigmondy set $\mathcal{Z}\left(\left\{\varphi^n(0)\right\}_{n \geqslant 1}\right)$ is empty.
\end{enumerate}
\end{thm}

This paper is organized as follows: In $\S 2,$ we will state some preliminaries which will be used in the proofs of our main results. In $\S 3,$ 
we shall give the proof of Theorems {\rm\ref{520002}}. In $\S 4,$ we shall give the proof of Theorem {\rm\ref{520004}}.
 In $\S 5,$ we shall give the proof of Theorem {\rm\ref{520005}}.\;
\section{\bf Preliminaries}

\subsection*{
\textbf{2.1 Dynamical System}
}
Let $K$ be a field and $\overline{K}$ be an algebraic closure of $K.$
		A rational function $\varphi(z)=\frac{f(z)}{g(z)} \in K(z)$ is a quotient of polynomials $f(z), g(z)\in K[z]$ with no common factors.
		 The degree of $\varphi$ is
		$
		\operatorname{deg} \varphi=\max \{\operatorname{deg} f, \operatorname{deg} g\}.
		$
		Let $\varphi^{\prime}(z)$ be the formal derivative of $\varphi(z).$ {Let $\alpha\in \overline{K}.$}
		If $\varphi^{\prime}(\alpha)=0,$ then ${\alpha}$ is called a critical point of $\varphi(z).$
		The rational function $\varphi$ of degree $d$ induces a rational map (morphism) of the projective space $\mathbb{P}^1(K),$
		$$
		\varphi: \mathbb{P}^1(K) \longrightarrow \mathbb{P}^1(K),\;\;
		\varphi([X:Y]) = \left[Y^d f(X / Y): Y^d g(X / Y)\right].$$
		We write
		$$
		\varphi^n=\underbrace{\varphi \circ \varphi \circ \cdots \circ \varphi}_{n \text { iterations }}
		$$
		for the $n$-th iterate of $\varphi.$
		The iterates of $\varphi$ applied to a point $P \in \mathbb{P}^1(K)$ give the forward orbit of $P,$ which we denote by
		$$
		\mathcal{O}_\varphi^{+}(P)=\left\{P, \varphi(P), \varphi^2(P), \varphi^3(P), \ldots\right\}.
		$$
		The point $P$ is called a wandering point of $\varphi$ if $\mathcal{O}_{\varphi}^{+}(P)$ is an infinite set; otherwise $P$ is called a preperiodic point of $\varphi.$
  The backward orbit $O_{\varphi}^{-}(P)$ under $\varphi$ is  the union over all $n \geqslant 0$ of $\varphi^{-n}(P):=\left\{Q \in \mathbb{P}^1\left(\overline{K}\right)\;|\; \varphi^n(Q)=P\right\}.$
  We say that a point $P \in \mathbb{P}^1(K)$ is an exceptional point if the backward orbit $O_{\varphi}^{-}(P)$ under $\varphi$ is finite. It is a standard fact that $P$ is an exceptional point of $\varphi$ if and only if $P$ is a totally ramified fixed point of $\varphi^2$ (See page 807 in \cite{SJH1}).

\subsection*{
\textbf{2.2 $S$-rigid Divisibility Sequence and  Primitive Divisor.}
}

\begin{Def}\label{00002}
			Let $K$ be a number field and let $\mathcal{A}=\left\{\mathfrak{A}_n\right\}_{n \geqslant 1}$ be a sequence of nonzero integral ideals.
			We can write $$ \mathfrak{A}_n=\mathfrak{p}_1^{m_1} \cdots \mathfrak{p}_r^{m_r}\mathfrak{q}_1^{t_1} \ldots \mathfrak{q}_{l}^{t_{l}},$$ 
			where $\mathfrak{p}_i \text {'s are the primitive prime  divisors of }\mathfrak{A}_n$ and $\mathfrak{q}_j $'s are the prime  divisors of $\mathfrak{A}_n$ which are not primitive.
			\\	Set:
			$$
			\begin{aligned}
				P_n & =\mathfrak{p}_1^{m_1} \cdots \mathfrak{p}_r^{m_r}\text { the primitive part of } \mathfrak{A}_n \text { and } \\
				{N}_n & =\mathfrak{q}_1^{t_1} \ldots \mathfrak{q}_{l}^{t_{l}}\text { the non-primitive part of } \mathfrak{A}_n.
			\end{aligned}
			$$
\end{Def}
		\begin{Def}\label{00008}
			Let $S$ be a finite set of places of the number field $K,$ including all archimedean places, and let $\mathfrak{A}$ be an integral ideal. The prime-to-S norm of $\mathfrak{A}$ is the quantity
			$$
			\mathcal{N}_S \mathfrak{A}=N_{K/\mathbb{Q}}\left(\prod_{\mathfrak{p} \notin S} \mathfrak{p}^{\operatorname{\rm ord}_{\mathfrak{p}}\mathfrak{A}}\right).
			$$	
		\end{Def}
		
	\begin{lem}\label{00003}
			If $\{\mathfrak{A}_n\}_{n\geqslant 1}$ is an $S$-rigid divisibility sequence, then
			\begin{equation*}
				\mathcal{N}_S N_{n} \leqslant \mathcal{N}_S \left( \prod_{i \mid n, i \neq n} P_i \right),\quad n\geq 2.
			\end{equation*}
		\end{lem}
\bpf It is sufficient to show that for any $\mathfrak{q} \notin S$, $\operatorname{ord}_{\mathfrak{q}} N_n \leq \operatorname{ord}_{\mathfrak{q}}\left(\prod\limits_{i \mid n, i \neq n} P_i\right)$.
Let $\mathfrak{q}$ be a prime ideal divisor in $N_n$ and $\mathfrak{q} \notin S$. Since $\mathfrak{q}$ is not a primitive prime ideal divisor of $\mathfrak{A}_n$, there exists a positive integer $d<n$ such that $\mathfrak{q}$ is a primitive prime ideal divisor of $\mathfrak{A}_d$. By the definition of $S$-rigid divisibility sequence, we obtain that $\operatorname{gcd}(n, d)=d$ and
$
\operatorname{ord}_{\mathfrak{q}} N_n=\operatorname{ord}_{\mathfrak{q}} \mathfrak{A}_n=\operatorname{ord}_{\mathfrak{q}} \mathfrak{A}_d=\operatorname{ord}_{\mathfrak{q}} P_d \leq \operatorname{ord}_{\mathfrak{q}}\left(\prod\limits_{i \mid n, i \neq n} P_i\right).
$
The last inequality follows from $d \mid n$ and $d<n$.
 \qed
\epf	
\subsection*{
\textbf{2.3 Height and Arithmetic Distance.}
}
Now let $K / \mathbb{Q}$ be a number field. The set of standard absolute values on $K$ is denoted by $M_K.$  We write $M_K^{\infty}$ for the archimedean absolute values on $K$ and $M_K^0$ for the nonarchimedean absolute values on $K.$ 
	For $v\in M_K,$ we also write $K_v$ for the completion of $K$ with respect to $|\cdot|_v,$ and we let $\mathbb{C}_v$ denote the completion of an algebraic closure of $K_v.$
 Let $P=\left[x_0, x_1\right] \in \mathbb{P}^1(K).$ The height of $P$ is
$$
h(P)=\sum_{v \in M_K} \frac{\left[K_v: \mathbb{Q}_v\right]}{[K: \mathbb{Q}]} \log \max \left(\left|x_0\right|_v,\left|x_1\right|_v\right).
$$

To simply notation, we let $$d_v=\frac{\left[K_v: \mathbb{Q}_v\right]}{[K: \mathbb{Q}]}.$$
For the definition of the height of a rational map, we  refer the reader to  \cite{HLC} or \cite{SJH2}.

		\blem\label{4001}
  Let $K$ be a number field. Let $\varphi\in K(z)$ be a rational function of degree $d\geqslant 2.$ Let $P\in \mathbb{P}^1\left(\overline{K}\right).$ Then
		the limit
		\begin{eqnarray}\label{036}
			\hat{h}_{\varphi}(P):=\lim _{n \rightarrow \infty} \frac{h\left(\varphi^n(P) \right)} { d^n}
		\end{eqnarray}
		exists and satisfies:
		\begin{itemize}
	\item[\rm(i)] There are constants \( c_3(d) \) and \( c_4(d) \) such that
\[
|\hat{h}_{\varphi}(P) - h(P)| \leqslant c_3(d) + c_4(d) h(\varphi)
\]
for all \( P \in \mathbb{P}^1\left(\overline{K}\right),\) where \( c_3(d) \) and \( c_4(d) \) depend only on \( d.\) Furthermore, expressions for \( c_3(d) \) and \( c_4(d) \) in terms of \( d \) can be found. 

			\item[\rm(ii)] $\hat{h}_{\varphi}(\varphi(P)) =d \hat{h}_{\varphi}(P)$  for all  $P \in \mathbb{P}^1\left(\overline{K}\right).$
			\item[\rm(iii)]
			$\hat{h}_{\varphi}(P)\geqslant 0,$ and 
			$\hat{h}_{\varphi}(P)>0$ if and only if  $P$ is a wandering point.
			
   	\item[\rm(iv)]
   Let $\phi: \mathbb{P}^1 \rightarrow \mathbb{P}^1$ be a morphism defined over $K$ and let $f \in \mathrm{PGL}_{2}(K)$ be an automorphism of $\mathbb{P}^1.$ \;Set \;$\phi^f=f^{-1} \circ \phi \circ f.$ Then
$$
\hat{h}_{\phi^ f}(P)=\hat{h}_\phi(f(P)) \quad \text { for all } P \in \mathbb{P}^1\left(\overline{K}\right).
$$
\item[\rm(v)] Let $\beta \in K^*$ and we write the ideal $(\beta)$ as a quotient of relatively prime integral ideals $(\beta)=\mathfrak{A} \mathfrak{B}^{-1}.$ Let $S$ be a finite set of places of $K,$ including all archimedean places. Then 
$$h(\beta)=\frac{1}{[K: \mathbb{Q}]}\left(\log \mathcal{N}_S \mathfrak{B}+\sum\limits_{v \in S}\left[K_v: \mathbb{Q}_v\right] \log \max \left\{1,|\beta|_v\right\}\right).$$
\end{itemize}
\elem
\bpf (i)  See \cite{HLC}, Proposition 6 (a), or see \cite{SJH2}, Exercise 3.8 and  page 98.\\
(ii), (iii), (iv) See \cite{SJH2}, Section 3.4 and Exercise 3.11.\\
(v) See \cite{IP}, page 296, (1.17).\qed
\epf
\begin{Def}\label{3001}
		Let $K$ be a number field, and let $P=\left[x_1: y_1\right]$ and $Q=\left[x_2: y_2\right]$ be points in $\mathbb{P}^1(\mathbb{C}_v).$ 
		
    \textnormal{(1) The $v$-adic chordal metric on $\mathbb{P}^1(\mathbb{C}_v)$ is defined by}
		$$
		\rho_v\left(P, Q\right) =
		\begin{cases}
		\displaystyle \frac{\left|x_1 y_2 - x_2 y_1\right|_v}{\sqrt{\left|x_1\right|_v^2 + \left|y_1\right|_v^2} \sqrt{\left|x_2\right|_v^2 + \left|y_2\right|_v^2}} & \textnormal{if } v \in \mathcal{M}_K^{\infty}, \\[15pt]
		\displaystyle \frac{\left|x_1 y_2 - x_2 y_1\right|_v}{\max \left\{\left|x_1\right|_v, \left|y_1\right|_v\right\} \max \left\{\left|x_2\right|_v, \left|y_2\right|_v\right\}} & \textnormal{if } v \in \mathcal{M}_K^0.
		\end{cases}
		$$
		
		\textnormal{(2) The function}
		$
		\lambda_v: \mathbb{P}^1\left(\mathbb{C}_v\right) \times \mathbb{P}^1\left(\mathbb{C}_v\right) \rightarrow \mathbb{R} \cup\{\infty\}
		$
		\textnormal{is defined by}
		$$\lambda_v\left(P,Q\right)=-\log \rho_v\left(P,Q\right).
		$$
		
	\end{Def}

  \blem\label{00005}
Let $K$ be a number field. Let $\infty = [1:0] \in \mathbb{P}^1(K).$ For any point $P  \in \mathbb{P}^1(K),$  we have
\[
h(P) \leqslant \sum_{v \in M_K} d_v \lambda_v(P, \infty).
\]
\elem
\begin{proof}
See \cite{HLC}, page 325. Alternatively,  the result can be directly established by applying the Product Formula.
\qed
\end{proof}

 \vskip 2mm
\section{\centering \textbf{Proof of Theorems {\rm\ref{520002}}}}

\blem\label{00004}
	Let $K$ be a number field and let $\varphi(z)\in K[z]$ of degree $d\geq 2.$ Assume that  $\alpha\in K$ is a wandering point of $\varphi$ and $\varphi^m(\alpha)\neq 0$ for all $m\geq 1.$
	For $n\geqslant 1,$ write the ideal  $(\varphi^n(\alpha))\mathcal{O}_K=\mathfrak{A}_n\mathfrak{B}_n^{-1}$  as a quotient of relatively prime integral ideals.
	Then we have 		$$
		\frac{1}{[K: \mathbb{Q}]} \log \mathrm{N}_{K / \mathbb{Q}} \mathfrak{A}_n \leqslant d^n \hat{h}_{\varphi}(\alpha)+c_3(d)+c_4(d)h(\varphi) \quad \text { for all } n \geqslant 1.
		$$
		where the constants $c_3(d)$ and $c_4(d)$   are determined by Lemma {\rm{\ref{4001}}} {\rm{(i).}}
	\elem
	\bpf
Note that $\varphi^m(\alpha)\neq 0$ for all $m\geq 1.$	We can define $\beta_n=\varphi^n(\alpha)^{-1}.$ Then $\beta_n\mathcal{O}_K=\mathfrak{B}_n\mathfrak{A}_n^{-1}.$
	In Lemma {\rm \ref{4001} (v)}, take $S=M_K^{\infty}$  and $\beta=\beta_n,$ we obtain that
	\begin{eqnarray}\label{038}
		h(\beta_n)\geq\frac{1}{[K:\mathbb{Q}]}\left(\log \mathrm{N}_{K / \mathbb{Q}} \mathfrak{A}_n\right).
	\end{eqnarray}
By Lemma {\rm{\ref{4001}}} {\rm{(i),}} we know that there exist constants $c_3(d)$ and $c_4(d)$ which only depend on $d$ such that

\begin{eqnarray}\label{040}
\begin{array}{ll}
h(\beta_n) &= h(\varphi^n(\alpha)^{-1}) 
    = h\left(\varphi^n(\alpha)\right) \\
    &\leq \hat{h}_{\varphi}\left(\varphi^n(\alpha)\right) + c_3(d) + c_4(d)h(\varphi) 
    \\&= d^n \hat{h}_{\varphi}(\alpha) + c_3(d) + c_4(d)h(\varphi).
    \end{array}
\end{eqnarray}
Combining \eqref{038} and \eqref{040}, we complete the proof.\qed
\epf

\blem\label{00006}
Let $K$ be a number field and let $\varphi(z)\in K[z]$ of degree $d\geq 2.$ Assume that  $\alpha\in K$ is a wandering point of $\varphi$ and $\varphi^m(\alpha)\neq 0$ for all $m\geq 1.$ Let \(S'\) be a finite set of places, including all archimedean ones. Let $T_n := [1 : \varphi^n(\alpha)] \in \mathbb{P}^1(K)$ and
$(\varphi^n(\alpha))\mathcal{O}_K=\mathfrak{A}_n\mathfrak{B}_n^{-1}$  as a quotient of relatively prime integral ideals.
Then
$$
\sum\limits_{v \in S'} d_v \lambda_v(T_n, \infty) \geqslant h(T_n) -\frac{1}{[K: \mathbb{Q}]} \log \mathcal{N}_{S'}(\mathfrak{A}_n).
$$
\elem
\bpf Since $\varphi^m(\alpha)\neq 0$ for all $m\geq 1,$ it follows that $T_n \neq[1$ : 0$]$ (i.e., $\infty$ ) for any $n \geqslant 1.$
Obviously,
if $v \in M_K^0,$ then $\lambda_v\left(T_n, \infty\right)=\log \max \left\{1,\left|\varphi^n(\alpha)\right|_v^{-1}\right\}.$

On the one hand, by Lemma {\rm\ref{00005}}, we obtain
$$
\begin{gathered}
\sum_{v \in S'} d_v \lambda_v\left(T_n, \infty\right)=\sum_{v \in M_K} d_v \lambda_v\left(T_n, \infty\right)-\sum_{v \notin S'} d_v \lambda_v\left(T_n, \infty\right) \\
\quad\quad\quad\quad\quad\quad\quad\quad\quad\geqslant h\left(T_n\right)-\sum_{v \notin S'} d_v \log \max \left\{1,|{\varphi^n(\alpha)}^{-1}|_v\right\}.
\end{gathered}
$$

On the other hand, by $\left(\varphi^n(\alpha)\right) O_K=\mathfrak{A}_n \mathfrak{B}_n^{-1},$ we get $\left(\varphi^n(\alpha)^{-1} \right)O_K=\mathfrak{B}_n \mathfrak{A}_n^{-1}.$\\
By {Lemma 
 \rm\ref{4001}\;(v)}, we obtain
$$
h\left(\varphi^n(\alpha)^{-1}\right)=\frac{1}{[K: \mathbb{Q}]} \log \mathcal{N}_{S'}\left(\mathfrak{A}_n\right)+\sum_{v \in S'} d_v \log \max \left\{1,\left|\varphi^n(\alpha)^{-1}\right|_v\right\}.
$$
By the definition of height, we have
$$
\frac{1}{[K: \mathbb{Q}]} \log \mathcal{N}_{S'}\left(\mathfrak{A}_n\right)=\sum_{v \notin S'} d_v \log \max \left\{1,\left|\varphi^n(\alpha)^{-1}\right|_v\right\}.
$$
Therefore,
\[
\sum_{v \in S'} d_v \lambda_v(T_n, \infty) \geqslant h(T_n) - \frac{1}{[K: \mathbb{Q}]} \log \mathcal{N}_{S'}(\mathfrak{A}_n).
\]
\qed\epf

\blem\label{00007}
Let $\varphi(z)\in K[z]$ of degree $d\geq2.$
 Assume that $0$ is a wandering point of $\varphi.$
Put \( Q_n := [1 : \varphi^n(0)], \;n\geqslant 1 \) and \(\infty = [1:0] \in \mathbb{P}^1(K).\) Let \( \sigma(z) = \frac{1}{z} \in \mathrm{PGL}_2(K) \) and \(\tilde{\varphi} = \sigma^{-1} \circ \varphi \circ \sigma.\) Then $$\hat{h}_{\tilde{\varphi}}(\infty)=\hat{h}_{\varphi}(0),\;
 \hat{h}_{\tilde{\varphi}}(Q_n) = d^n \hat{h}_{\varphi}(0) \;\text{and}\; Q_n = \tilde{\varphi}^n(\infty).$$
\elem

\begin{proof}
By Lemma {\rm\ref{4001}} (iv),
we have $\hat{h}_{\tilde{\varphi}}(\infty)=\hat{h}_{\varphi}(0)$ and $
\hat{h}_{\tilde{\varphi}}(Q_n) = \hat{h}_{\varphi}(\sigma(Q_n)).$
 It is clear that
$
\hat{h}_{\varphi}(\sigma(Q_n)) = \hat{h}_{\varphi}(\varphi^n(0)) = d^n \hat{h}_{\varphi}(0).$
Hence \( \hat{h}_{\tilde{\varphi}}(Q_n) = d^n \hat{h}_{\varphi}(0). \)

It is easy to obtain that
\(
\begin{aligned}
Q_n 
= \sigma^{-1} \circ \varphi^n \circ \sigma([1: 0]) 
= \tilde{\varphi}^n(\infty).
\end{aligned}
\)
\qed
\end{proof}
\begin{lem}\label{330}
With the notation and assumptions as Lemma  {\rm\ref{00007}}.  Let \(S'\) be a finite set of places, including all archimedean ones. Set
$$
J(S',\varphi) = \left\{ n \in \mathbb{N}^* \;\bigg| \;\sum\limits_{v \in S'} d_v \lambda_v\left(Q_n, \infty\right) \geqslant \frac{1}{8} \hat{h}_{\tilde{\varphi}}\left(Q_n\right) \right\}.
$$

Then there exists a constant \( \gamma,\) depending only on \( d \) and \( [K : \mathbb{Q}],\) such that
\[
\#J(S',\varphi) \leqslant 4^{\#S'} \gamma + \log_d\left(\frac{h(\tilde{\varphi})}{\hat{h}_{\varphi}(0)} + 1\right).
\]
\end{lem}
\bpf
Since \( 0 \) is a wandering point of \( \varphi,\) we know that \(\infty = [1:0]\) is a wandering point of \(\tilde{\varphi}.\) Therefore, \(\infty\) is not an exceptional point of \(\tilde{\varphi}.\)

Applying Theorem 11(b) in \cite{HLC} for
 \( A = [1:0] \) (i.e., \(\infty\)),\;\( P = [1:0],\)\;$\epsilon_0=\frac{1}{8},$ \(\tilde{\varphi}\) and $S',$ we obtain that 
there exists a constant \(\gamma\) depending only on \(d\) and \([K : \mathbb{Q}]\) such that
\[\begin{array}{l}
\;\;\;\;\#\left\{n \in \mathbb{N}^* \;\bigg|\; \sum_{v \in S'} d_v \lambda_v\left(Q_n, \infty\right) \geqslant \frac{1}{8} \hat{h}_{\tilde{\varphi}}\left(Q_n\right)\right\} \\[9pt]
=
\#\left\{n \in \mathbb{N}^* \;\bigg|\; \sum_{v \in S'} d_v \lambda_v\left(\tilde{\varphi}^n(\infty), \infty\right) \geqslant \frac{1}{8} \hat{h}_{\tilde{\varphi}}\left(\tilde{\varphi}^n(\infty)\right)\right\}\\ [8pt]
\leqslant 4^{\#S'} \gamma + \log_d\left(\frac{h(\tilde{\varphi})}{\hat{h}_{\varphi}(0)} + 1\right).
\end{array}
\]
\qed\epf

Assume that \(0\) is a wandering point of \(\varphi.\)  Define

\[
X(\varphi) = \left\{ n \in \mathbb{N} \;\middle|\; n \leqslant \log_d^{+} \left(\frac{c_3(d) + c_4(d) h(\varphi)}{\frac{1}{8} \widehat{h}_{\varphi}(0)}\right) \right\},
\]
\vskip 1.5mm
\[
I(\varphi) = \left\{ n \in \mathbb{N} \;\middle|\; \frac{(n-1)(c_3(d) + c_4(d) h(\varphi)) + \widehat{h}_{\varphi}(0) \frac{d^n - d}{d - 1}}{\frac{3}{4} \widehat{h}_{\varphi}(0) d^n} \geqslant 1 \right\},
\]
where the constants $c_3(d)$ and $c_4(d)$   are determined by Lemma {\rm{\ref{4001}}} {\rm{(i).}}

\blem\label{000081}
With the notation and assumptions as Lemmas {\rm\ref{00007}} and  {\rm\ref{330}}. Let $(\varphi^n(0))\mathcal{O}_K=\mathfrak{A}_n\mathfrak{B}_n^{-1}.$
If \( n \notin X(\varphi)\cup J(S', \varphi),\) then
\[
\frac{3}{4} \hat{h}_{\varphi}(0) d^n < \frac{1}{[K: \mathbb{Q}]} \log \mathcal{N}_{S'}(\mathfrak{A}_n).
\]
\elem

\bpf
Since \( n \notin J(S',\varphi),\) we have
\[
\sum_{v \in S'} d_v \lambda_v\left(Q_n, \infty\right) < \frac{1}{8} \hat{h}_{\tilde{\varphi}}\left(Q_n\right).
\]
Applying Lemma {\rm\ref{00006}} for $\alpha = 0$, and noting that \( \operatorname{deg} \varphi = \operatorname{deg} \tilde{\varphi},\) as well as Lemma {\rm\ref{00007}}, we obtain
\[
h(Q_n) - \frac{1}{[K: \mathbb{Q}]} \log \mathcal{N}_{S'}(\mathfrak{A}_n) < \frac{1}{8} d^n \hat{h}_{\varphi}(0).
\]
From
$h(\varphi^n(0)) > \hat{h}_{\varphi}(\varphi^n(0)) - c_3 (d) - c_4 (d) h(\varphi)$
and 
$\hat{h}_{\varphi}(\varphi^n(0)) = d^n \hat{h}_{\varphi}(0),$
we get
\[
h(\varphi^n(0)) > d^n \hat{h}_{\varphi}(0) - c_3 (d) - c_4 (d) h(\varphi).
\]

Hence
\[
\frac{7}{8} d^n \hat{h}_{\varphi}(0) - c_3(d) - c_4(d) h(\varphi) < \frac{1}{[K: \mathbb{Q}]} \log \mathcal{N}_{S'}(\mathfrak{A}_n).
\]
By \( n \notin X(\varphi),\) we obtain
$\frac{-c_3(d) - c_4(d) h(\varphi)}{d^n} > -\frac{1}{8} \hat{h}_{\varphi}(0).$
Therefore,
\[
\frac{3}{4} \hat{h}_{\varphi}(0) d^n < \frac{1}{[K: \mathbb{Q}]} \log \mathcal{N}_{S'}(\mathfrak{A}_n).
\]
\qed
\epf

\subsection*{
\textbf{Proof of Theorem \ref{520002}.}}
Set $\psi(z)=\varphi(z+\alpha)-\alpha.$ Then  we have
$$
\varphi^n(\alpha)-\alpha=\psi^n(0), n \geqslant 1.
$$
Hence  $0$ is a wandering point of $\psi(z).$ 
Obviously, $\operatorname{deg}(\psi)=\operatorname{deg}(\varphi)=d.$\;
First, we claim that if
$n \notin I(\psi)\cup J(S,\psi)\cup X(\psi),$ then $\mathfrak{A}_n $ has a primitive divisor.

The primitive part of \( \mathfrak{A}_n \) is denoted by \( P_n,\) and the non-primitive part is denoted by \( N_n.\)
Note that
$$
\begin{aligned}
&\log \mathcal{N}_S\left(N_n\right)\leqslant \log\mathcal{N}_S \left( \prod_{i \mid n, i \neq n} P_i \right)
\leqslant \sum_{j=1}^{n-1} \log \mathcal{N}_S\left(\mathfrak{A}_j\right)
\leqslant \sum_{j=1}^{n-1} \log N_{K/\mathbb{Q}}\left(\mathfrak{A}_j\right)
\\
& \leqslant[K:\mathbb{Q}]\left((n-1)\left(c_3(d)+c_4(d) h(\psi)\right)+\hat{h}_{\psi}(0) \frac{d^n-1}{d-1}\right) 
<\frac{3}{4}[K: \mathbb{Q}] \hat{h}_{\psi}(0) d^n \\
& <\log \mathcal{N}_S\left(\mathfrak{A}_n\right).\\& \end{aligned}
$$\vspace{-2em}

The first inequality comes from Lemma {\rm\ref{00003}.}
The second inequality comes from \( P_j \mid \mathfrak{A}_j.\)
The third inequality comes from the Definition {\rm\ref{00008}.}
The fourth inequality comes from Lemma {\rm\ref{00004}.}
The fifth inequality comes from \( n \notin I(\psi).\)
The sixth inequality comes from  Lemma {\rm\ref{000081}}.
Hence
$$
\log \mathcal{N}_S\left(P_n\right)=\frac{\log \mathcal{N}_S\left(\mathfrak{A}_n\right)}{\log \mathcal{N}_S\left(N_n\right)}>1.\;
$$
Therefore $P_n $ is not trivial,\;i.e.,\;\;$\mathfrak{A}_n$ \;has a primitive divisor. This completes the proof of the claim.
Therefore, we have
$$\#\mathcal{Z}(\{\mathfrak{A}_n\}_{n \geqslant 1}) \leqslant  \# I(\psi) +\# J(S,\psi)+ \# X(\psi).$$
Note that $d\geqslant 3.$
A bit of algebraic calculation implies that
\[
\#I(\psi) \leqslant 1+\frac{8(c_3(d) + c_4(d) h(\psi))}{(3d-7) \hat{h}_{\psi}(0)}.
\]
 So, we can take $M$ as follows:
\[
1 + \log_d^+ \left( \frac{8(c_3(d) + c_4(d) h(\psi))}{\widehat{h}_{\psi}(0)} \right) 
+ \frac{8(c_3(d) + c_4(d) h(\psi))}{(3d-7) \hat{h}_{\psi}(0)} 
+ 4^{\#S} \gamma + \log_d \left( \frac{h(\tilde{\psi})}{\hat{h}_{\psi}(0)} + 1 \right).
\]
 \qed
\section{\bf Proof of Theorem {\rm\ref{520004}} }
Let  \(\varphi(z) = f_1(z)^{e_1} f_2(z)^{e_2} \cdots f_m(z)^{e_m}\) be a powerful polynomial over $K.$ If the \( i \)-th coefficient of \( f_j(z) \) is non-zero, we denote it by \( a^{(j)}_i.\) Write \( a^{(j)}_i \mathcal{O}_K = I^{(j)}_i (J^{(j)}_i)^{-1},\) where \( I^{(j)}_i \) and \( J^{(j)}_i \) are relatively prime integral ideals.\\
Put 
$$
		S\!=S(f_1,f_2,\ldots,f_m) \!=\! \left\{\text{prime ideal } \bm{\beta} \;\bigg|\;\bm{\beta} \mid J^{(j)}_i \text{ for some } 1 \leqslant j \leqslant m \text{ and } i\right\} \cup M_K^{\infty}.$$
Let $\mathcal{O}_{K, S}$  be the ring of  $S$ -integers given by $$ \mathcal{O}_{K, S}=\left\{x \in K|\operatorname{\rm ord}_{\mathfrak{q}}(x)\geqslant 0 \text { for all prime ideal}\; \mathfrak{q} \notin S\right\}.$$
It is obvious that $f_j(z)\in \mathcal{O}_{K, S}[z]$ for any $1\leqslant j\leqslant m.$ Hence $\varphi^n(0) \in \mathcal{O}_{K, S}$ for any $n \geqslant 1.$ 
		In light of the facts that $\mathfrak{A}_n$ and $\mathfrak{B}_n$ are relatively prime integral ideals,  we conclude that
		$$\operatorname{\rm ord}_{\mathfrak{p}}{(\varphi^{n}(0))}=\operatorname{\rm ord}_{\mathfrak{p}}{\mathfrak{A}_n}, \;\;\;\text{for any}\; n\in\mathbb{N}^*\;\text{and}\;\mathfrak{p} \notin S.$$
		Let $\mathfrak{p}$ be  a prime ideal such that $\mathfrak{p} \notin S$ and $\operatorname{ord}_{\mathfrak{p}}(\varphi^{n_{0}}(0)) > 0$ for some $n_0 \in \mathbb{N}^*.$
  
  Let $r = \min\{m \in \mathbb{N}^* \mid \operatorname{ord}_{\mathfrak{p}}(\varphi^{m}(0)) > 0\}.$ Write
		\begin{align}\label{thm3.7-2}
			\varphi^r(z) &= z g_r(z) + \varphi^r(0),
		\end{align}
		where \( g_r(z) \in \mathcal{O}_{K, S}[z].\)
Set \( E = \max\limits_{1 \leqslant j \leqslant m} e_j,\) we obtain that
\[
\varphi(z) \mid (\varphi'(z))^E \;\text{in the ring} \; \mathcal{O}_{K, S}[z].
\]
Hence
$
\varphi^{r}(0)\mid (\varphi'(\varphi^{r-1}(0)))^E$ in $\mathcal{O}_{K, S}.
$
(Note that $0$ is a wandering point of $\varphi.$)

Let \( g_r(z) = c_0 + c_1 z + \cdots + c_d z^d.\) Then \\
		
		\( c_0 = g_r(0) = \left. \left( \varphi^r(z) \right)' \right|_{x=0} = \varphi'(0) \varphi'(\varphi(0)) \cdots \varphi'(\varphi^{r-1}(0)).\)\\
		Hence, $$\operatorname{ord}_{\mathfrak{p}}(c_0)\geqslant \operatorname{ord}_{\mathfrak{p}}(\varphi'(\varphi^{r-1}(0)))\geqslant \frac{1}{E}\operatorname{ord}_{\mathfrak{p}}(\varphi^{r}(0))>0.$$
		
	On the one hand,	for any $j\in\mathbb{N}^{*},$ 
 $$\varphi^{jr}(0)=\varphi^r(\varphi^{(j-1)r}(0))\xlongequal[]{\eqref{thm3.7-2}}\varphi^{(j-1)r}(0)g_r(\varphi^{(j-1)r}(0)))+\varphi^r(0).$$
Hence, by induction on $j,$ we have
\begin{equation}\label{thm3.7-3}
  \operatorname{ord}_{\mathfrak{p}}\left(\varphi^{jr}(0)\right) = \operatorname{ord}_{\mathfrak{p}}\left(\varphi^{r}(0)\right) \quad \text{for any} \quad j \in \mathbb{N}^{*}.
\end{equation}

\noindent If $k < r,$ then the minimality of $r$ implies that $\operatorname{ord}_{\mathfrak{p}}(\varphi^{k}(0)) = 0.$ If $k>r$ and $r \nmid k,$ then $k = qr + l$ with $0 < l < r$ and $q \geqslant 1.$ 
Let 
$\varphi^l(z) = z g_l(z) + \varphi^l(0),$ where \( g_l(z) \in \mathcal{O}_{K, S}[z].\)
		Then we have
\begin{eqnarray*}\varphi^k(0)=\varphi^l\left(\varphi^{qr}(0)\right)=\varphi^{qr}(0) g_l\left(\varphi^{qr}(0)\right)+\varphi^l(0).\end{eqnarray*}

From \eqref{thm3.7-3}, we have $\operatorname{ord}_{\mathfrak{p}}(\varphi^{qr}(0)) > 0.$ Note that the minimality of $r$ implies that $\operatorname{ord}_{\mathfrak{p}}(\varphi^{l}(0)) = 0.$  Hence, we conclude that $\operatorname{ord}_{\mathfrak{p}}(\varphi^{k}(0)) = 0.$ 
	
On the other hand,
{\rm{}
let \( m, n \geqslant 1,\) and let \(\mathfrak{p}\) be a prime ideal of $\mathcal{O}_{K}$  with \(\mathfrak{p} \notin S\) and \(\mathfrak{p} \mid \operatorname{gcd}\left(\mathfrak{A}_n, \mathfrak{A}_m\right).\) 

By \(\operatorname{ord}_{\mathfrak{p}}(\mathfrak{A}_n) > 0\) and \(\operatorname{ord}_{\mathfrak{p}}(\mathfrak{A}_m) > 0,\) we have \( r \mid n \) and \( r \mid m,\) and so \( r \mid \operatorname{gcd}(n, m).\) Therefore,
\[
\operatorname{ord}_{\mathfrak{p}}(\mathfrak{A}_r) = \operatorname{ord}_{\mathfrak{p}}(\mathfrak{A}_{\operatorname{gcd}(n, m)}) = \operatorname{ord}_{\mathfrak{p}}(\mathfrak{A}_n) = \operatorname{ord}_{\mathfrak{p}}(\mathfrak{A}_m).
\]

 Hence, the sequence \(\{\mathfrak{A}_n\}_{n \geqslant 1}\) forms an \( S \)-rigid divisibility sequence.

By Theorem~\ref{520002}, we conclude that \(\mathcal{Z}(\{\mathfrak{A}_n\}_{n \geqslant 1})\) is a finite set.
\qed
}

\section{\bf Proof of Theorem {\rm\ref{520005}} }		
\begin{lem}\label{1650}
Let $m \geqslant 2,$ and $\varphi(z) =\prod\limits_{i=1}^m (zf_i(z) + a_i)^{e_i},$ where $a_1, a_2, \ldots, a_m\in\Z,$   $a_1a_2\cdots a_m\neq 0,\pm 1,$ and $e_i \geqslant 2,$  $f_i(z) \in \mathbb{Z}[z]$  has no integer roots,  $i = 1, 2, \dots, m.$
 Then $
\left|\varphi^n(0)\right| > \left|\varphi^{n-1}(0)\right|^2\geqslant \left|\varphi(0)\right|^2\geqslant 4$ for all $n \geqslant 2.$\end{lem}

\begin{proof} 
First, we claim that 
$\left|\varphi^n(0)\right| \geqslant \max\limits_{1\leqslant j\leqslant m}\{|a_j|^{\alpha_n}\},$  where $\alpha_n=\frac{2^n(m-1)m^{n-1}+2m}{2m-1},$  $n \geqslant 1.$
We prove the claim by induction on $n.$  For $n=1,$ we have
$$|\varphi(0)|=\prod\limits_{i=1}^m|a_i|^{e_i}\geqslant\max\limits_{1\leqslant j\leqslant m}\{|a_j|^2\}=\max\limits_{1\leqslant j\leqslant m}\{|a_j|^{\alpha_1}\}.$$

Suppose $n\geqslant 2$ and $|\varphi^{n-1}(0)|\geqslant \max\limits_{1\leqslant j\leqslant m}\{|a_j|^{\alpha_{n-1}}\}.$  Then 
\begin{eqnarray}\label{lemma5.1-1}
\begin{aligned}
    \left|\varphi^n(0)\right| &= \prod_{i=1}^m \left|\varphi^{n-1}(0)f_i(\varphi^{n-1}(0)) + a_i\right|^{e_i}\\
    &\geqslant \prod_{i=1}^m \left(\left|\varphi^{n-1}(0)\right|-\left|a_i\right|\right)^{e_i} \\
         &\geqslant \prod_{i=1}^m \left(\max\limits_{1\leqslant j\leqslant m}\{|a_j|^{\alpha_{n-1}}\}-
         \max\limits_{1\leqslant j\leqslant m}\{|a_j|^{\alpha_{n-1}-1}\}\right)^{e_i} \\
    &\geqslant \prod_{i=1}^m\left( \max\limits_{1\leqslant j\leqslant m}\{|a_j|^{\alpha_{n-1}-1}\}\right)^{e_i}\\
    & \geqslant \max\limits_{1\leqslant j\leqslant m}\{|a_j|^{2m(\alpha_{n-1}-1)}\}\\
    &=\max\limits_{1\leqslant j\leqslant m}\{|a_j|^{\alpha_{n}}\}.
\end{aligned}
\end{eqnarray}
This completes the proof of the claim.

For any $n\geqslant 2,$ $1\leqslant i\leqslant m,$ it is obvious that 
\begin{eqnarray}\label{lemma5.1-2}\left|\varphi^{n-1}(0)\right|\geqslant \max\limits_{1\leqslant j\leqslant m}\{|a_j|^{\alpha_{n-1}}\}\geqslant \max\limits_{1\leqslant j\leqslant m}\{|a_j|^{2}\}\geqslant \max\{|a_i|^2,4\}.\end{eqnarray}
Hence, we have 
\begin{eqnarray}\label{lemma5.1-3}
\begin{aligned}
    \left|\varphi^n(0)\right| &= \prod_{i=1}^m \left|\varphi^{n-1}(0)f_i(\varphi^{n-1}(0)) + a_i\right|^{e_i}\\
    &\geqslant \prod_{i=1}^m \left(\left|\varphi^{n-1}(0)\right|-\left|a_i\right|\right)^{e_i}\\
    &\geqslant \prod_{i=1}^m \left(\left|\varphi^{n-1}(0)\right|-\sqrt{\left|\varphi^{n-1}(0)\right|}\right)^{e_i}\\
    & \geqslant \left(\sqrt{\left|\varphi^{n-1}(0)\right|}\right)^{2m}\\
    &\geqslant \left|\varphi^{n-1}(0)\right|^{2}.
         \end{aligned}
\end{eqnarray}

  It is clear that  $\alpha_i>\alpha_1=2$ for any $i\geqslant 2.$ By \eqref{lemma5.1-1}, \eqref{lemma5.1-2} and \eqref{lemma5.1-3},  we have 
 $\left|\varphi^n(0)\right|> \left|\varphi^{n-1}(0)\right|^{2}$  if $n\geqslant 3$ or $m\geqslant 3$ or $|a_i|\neq |a_j|$ for some $1\leqslant i\neq j\leqslant m.$ 

If $m=2$ and $|a_1|=|a_2|\geqslant 2, e_1\geqslant 2, e_2\geqslant 2,$  then $|\varphi(0)| =|a_1^{e_1}a_2^{e_2}|>|a_1|^2=|a_2|^2.$ From \eqref{lemma5.1-3}, we have  $|\varphi^2(0)| > |\varphi(0)|^2.$
\qed
\end{proof}
\subsection*{
\textbf{Proof of Theorem \ref{520005}.}}
{\rm{
(1) It is trivial  that if $0$ is a fixed point of $\varphi(z),$ then $0$ is a preperiodic point of $\varphi(z).$ 

Assume that $0$ is not a fixed point of $\varphi(z).$ It is obvious that $a_1, a_2, \ldots, a_m$ are non-zero integers. By Lemma {\ref{1650}},  $
\left|\varphi^n(0)\right| > \left|\varphi^{n-1}(0)\right|^2\geqslant \left|\varphi(0)\right|^2 \geqslant 4$  for all  $n \geqslant 2.$ 
Hence, $
\left|\varphi^n(0)\right| > \left|\varphi^{n-1}(0)\right|$ for all $n\geq2.$
So, $0$ is not a preperiodic point of $\varphi(z).$

(2) Since $
\left|\varphi^n(0)\right| > \left|\varphi^{n-1}(0)\right|^2$ for all $n\geq2,$ we have $\prod\limits_{k=1}^{n-1}\left|\varphi^k(0)\right|<\left|\varphi^n(0)\right|.$
In Theorem {\rm\ref{520004}}, we take $K=\mathbb{Q}, S=M_\mathbb{Q}^{\infty}.$
Then $\{\varphi^n(0)\}_{n\geqslant 1}$ is an $S$-rigid divisibility sequence.
Let $P_n$ be the primitive part of $\varphi^n(0)$
and $N_n$ be the non-primitive part of $\varphi^n(0).$ 
By Lemma {\rm\ref{00003}}, 
$$
|N_n|\leq\prod_{d \mid n, d \neq n} |P_d| \leqslant \prod_{k=1}^{n-1} |P_k| \leqslant \prod_{k=1}^{n-1}|\varphi^k(0)|<\left|\varphi^n(0)\right|, n\geqslant 2.
$$
Hence $P_n$ is not trivial for $n\geqslant 2.$
This completes the proof of (2).
\qed
}}

\end{document}